\documentclass[11pt]{article}

 \usepackage{amsmath,amssymb,theorem}

\usepackage{color}

\numberwithin{equation}{section}

\newcommand {\h}{\mathbb{H}}
\newcommand {\Z}{\mathbb{Z}}

\newcommand {\R}{\mathbb{R}}
\newcommand {\C}{\mathbb{C}}

\newcommand{\alg}[1]{\mathfrak{#1}}

\newcommand{\affine}[1]{\widehat{#1}}

\newcommand{\SLA}[2]{\alg{#1} \left( #2 \right)}
\newcommand{\SLSA}[3]{\alg{#1} \left( #2 \middle\vert #3 \right)}
\newcommand{\AKMA}[2]{\affine{\alg{#1}} \left( #2 \right)}
\newcommand{\AKMSA}[3]{\affine{\alg{#1}} \left( #2 \middle\vert #3 \right)}

\newcommand{\TypMod}[1]{\mathcal{T}_{#1}}
\newcommand{\AtypMod}[1]{\mathcal{A}_{#1}}
\newcommand{\eop}{\hspace*{\fill} $\Box$}

\newtheorem{theorem}{Theorem}
\newtheorem{lemma}[theorem]{Lemma}
\newtheorem{proposition}[theorem]{Proposition}
\newtheorem{corollary}[theorem]{Corollary}
\newtheorem{conjecture}[theorem]{Conjecture}
\newtheorem{remark}[theorem]{Remark}
\newtheorem{definition}[theorem]{Definition}

\newenvironment{proof}[1][Proof]{\begin{trivlist}
\item[\hskip \labelsep {\bfseries #1}]}{\end{trivlist}}

\newcommand{\qed}{\nobreak \ifvmode \relax \else
      \ifdim\lastskip<1.5em \hskip-\lastskip
      \hskip1.5em plus0em minus0.5em \fi \nobreak
      \vrule height0.75em width0.5em depth0.25em\fi}

\begin{document}
\title{\bf The Mock Modular Data of a Family of Superalgebras}

\author{Claudia Alfes\footnote{alfes@mathematik.tu-darmstadt.de}\ \ and\ Thomas Creutzig\footnote{tcreutzig@mathematik.tu-darmstadt.de}\\[3mm]
Fachbereich Mathematik, \\ Technische Universit\"at Darmstadt, \\ 64289 Darmstadt, Germany} \maketitle

\abstract{The modular properties of characters of representations of a family of W-superalgebras extending $\AKMSA{gl}{1}{1}$ are considered.
Modules fall into two classes, the generic type and the non-generic one.
Characters of non-generic modules are expressed in terms of higher-level Appell-Lerch sums. We compute the modular transformations
of characters and interpret the Mordell integral as an integral over characters of generic representations. 
The $\C$-span of a finite number of non-generic characters together with an uncountable set of characters of the generic type combine into a representation
of $SL(2;\Z)$. The modular transformations are then used to define a product on the space of characters. 
The fusion rules of the extended algebras are partially inherited from the known fusion rules for modules of $\AKMSA{gl}{1}{1}$. Moreover, the product obtained from the modular transformations coincides with
the product of the Grothendieck ring of characters if and only if the fusion multiplicities are at most one.
}

\section{Introduction}

Automorphic forms, representation theory and conformal field theory are closely related areas since characters of modules of conformal field theories possess modular properties. 
Important examples are representations of affine Lie algebras and their characters as well
as lattice vertex operator algebras whose characters involve the theta function of the lattice. A crucial ingredient of the algebraic structure of (rational) conformal field theories is the fusion ring of modules, which is related to the modular transformations of characters via the Verlinde formula.

Here, we consider a family of W-superalgebras extending $\AKMSA{gl}{1}{1}$ which is a supersymmetric analogue of $\AKMA{gl}{1}$ and rank one lattice vertex algebras. To illustrate the main ideas of this construction, the investigation of modularity, and the Verlinde formula we start by presenting the well-known example of rank one lattice vertex algebras.
For an introduction to conformal field theory and a detailed exposition of this example see \cite{G}.
The affine vertex algebra of $\SLA{gl}{1}$ or rank one Heisenberg algebra has the inverse of Dedekind's $\eta$-function as its character. 
This algebra admits a family of extensions to larger conformal field theories as follows. 
Irreducible modules $V_\alpha$ of highest-weight $\alpha\in\R$ possess rather trivial fusion products, namely
\begin{equation}\label{eq:fusionex}
V_\alpha\times_{f} V_\beta\ = \ V_{\alpha+\beta}.
\end{equation}
This implies that the module
\begin{equation*}
M_\alpha \ = \ \bigoplus_{m\in\Z}V_{m\alpha}
\end{equation*}
is closed under the fusion product. Let $\alpha^2\in\mathbb Z$, then such a module extends the rank one Heisenberg algebra to a larger conformal field theory. 
Important examples of these extensions are the affine vertex algebra of $\SLA{sl}{2}$ at level one ($\alpha=\sqrt{2}$), and the
super Virasoro algebra at central charge one ($\alpha=\sqrt{3}$).
The character of the module $V_\alpha$ is
\begin{equation*}
\chi_{V_{\alpha}}(u;\tau)\ = \ \frac{e^{2\pi i \alpha u}e^{\pi i \tau \alpha^2}}{\eta(\tau)}\, ,
\end{equation*}
where $\tau\in\mathbb{H}$, $u\in\C$.
Hence, the character of the module $M_\alpha$ is 
\begin{equation*}
\chi_{M_{\alpha}}(u;\tau)\ = \ \frac{1}{\eta(\tau)}\sum_{m\in\Z}e^{2\pi i \alpha m u}e^{\pi i \tau \alpha^2 m^2}\ = \ \frac{\theta_3(\alpha u;\alpha^2\tau)}{\eta(\tau)}\, .
\end{equation*}
If $\alpha$ is the square root of a positive even integer, these characters have particularly good modular properties. 
The extended algebra is the lattice $\alpha\Z$.
The modules 
\begin{equation*}
M_{\alpha,n}\ = \ \bigoplus_{m\in\Z}V_{m\alpha+n/\alpha}
\end{equation*}
for $n=0,...,\alpha^2-1$ close under the fusion product and its characters form
a representation of the modular group $\text{SL}(2;\Z)$.

Here we investigate a supersymmetric analogue of this simple example.
Consider the affine Lie superalgebra $\AKMSA{gl}{1}{1}$. The character of its affine vertex superalgebra is
\begin{align*}
 \chi_{\AKMSA{gl}{1}{1}}(u,v;\tau) \ &= \ i \frac{e^{i\pi u}}{1-e^{2\pi i u}}\frac{1}{\eta^3(\tau)} \left(-i \sum_{n\in\Z} (-1)^n e^{\pi i (n+\frac12)^2 \tau+2\pi i u (n+\frac12)}\right) 
\\
&= \ i \frac{e^{i\pi u}}{1-e^{2\pi i u}}\frac{\theta_1(u;\tau)}{\eta^3(\tau)}\, ,
\end{align*}
where $\tau\in\mathbb{H}$, $u,v\in\C$. Verma modules of this algebra are generically irreducible. The non-generic case is called atypical and gives rise to interesting extended algebras \cite{CR1}.
Atypical modules are labeled by a pair $(n,\ell)$, where $\ell$ is an integer and $n$ a real number.
Their characters are, for example if $\ell>0$,
\begin{equation*}
\begin{split}
\chi_{\AtypMod{n,\ell}}(u,v;\tau) \ &= \ (-1)^{\ell+1} i e^{2\pi iv\ell} \frac{e^{\pi i(2nu+(2n\ell+\ell^2)\tau)}}{1-e^{2\pi i(un+\ell\tau)}}\frac{\theta_1(u;\tau)}{\eta^3(\tau)}\, .
\end{split}
\end{equation*}
The fusion product of atypical modules is not complicated \cite{CR1}. 
As in the $\AKMA{gl}{1}$ case, there exist sets of modules that close under fusion and thus combine to an extended algebra.
We are considering a class that is labeled by two integers $n,\ell$, such that $\ell>0$, $n=0$ mod $\ell$ and $n\ell+\ell^2>0$. We call these extended algebras $\mathfrak W_{n,\ell}$.
The most interesting family in this class is obtained by setting $\ell=1$ and $n=0,1,2,...$\ . In \cite{CR2} it was shown that the operator product algebra of an even subalgebra of $\mathfrak W_{n,\ell}$ agrees with
the operator product algebra of the Feigin-Semikhatov $\mathbf W^{(2)}_{2n+1}$ algebra at level $-2(n-1)(2n+1)/(2n-1)$ \cite{FS}.
The character of such an extended algebra is 
\begin{equation*}
\chi_{\mathfrak W_{n,\ell}}(u,v;\tau)\ = \ i\frac{\theta_1(u;\tau)}{\eta^3(\tau)}\sum_{m\in\Z}\frac{e^{\pi i (m\ell+(2mn+1)u+2mv+m\ell(m\ell+2mn+1)\tau)}}{1-e^{2\pi i (u+m\ell\tau)}}\, .
\end{equation*}
This character can be expressed in terms of an higher-level Appell-Lerch sum.
Using the modularity of these sums we find analogously to $\AKMA{gl}{1}$ and rank one lattice vertex algebras
\begin{theorem}
Let $n,\ell\in \Z$, such that $\ell>0$, $n\equiv0\pmod{\ell}$ and $n\ell+\ell^2>0$. Then there exists a set of modules of the algebra $\mathfrak W_{n,\ell}$, such that the $\C$-span of its characters is a representation of the modular group $\text{SL}(2;\Z)$. 
\end{theorem} 

The set of characters of modules appearing involves a finite subset of non-generic, we call it again atypical, modules whose characters are expressed in terms of Appell-Lerch sums, as well as a set of uncountable size of generic or typical ones. We are able to relate the latter to the shifted Mordell integral appearing in the $S$-transformation (the fractional linear transformation given by the involution $\left(\begin{smallmatrix} 0 &-1\\1&0\end{smallmatrix}\right)$ of the modular group $\mathrm{SL}_2(\Z)$) of the Appell-Lerch sum thus giving a representation theoretical interpretation of this term. A phenomenon of similar type has been noticed in string theory \cite{M, ET}.

A connection between the representation theory of affine Lie superalgebras, Appell-Lerch sums, 
and number theory was first noticed by Kac and Wakimoto \cite{Kac2001} (see also \cite{Kac1984, Kac1994}). 
They discovered that characters of affine Lie superalgebras at level one are related to Ramanujan's mock theta functions \cite{Kac2001}. 
Using work of Zwegers \cite[Chapter 1]{Zwegers2002} on Appell-Lerch sums and mock theta functions Bringmann and Ono \cite{BriOno} were 
able to show that some of the $q$-series obtained by Kac and Wakimoto are the ``holomorphic'' parts of some non-holomorphic functions 
which are products of weight $1/2$ harmonic Maass forms (introduced by Bruinier and Funke \cite{BruinierFunke}) with a weight $-1/2$ quotient of $\eta$-functions. Recently, characters of other Lie superalgebras have been related to Appell-Lerch sums (for 
example \cite{Schomerusetc}, \cite{Folsom}, \cite{BringmannFolsom}). Zwegers \cite{Zwegers1} also embedded other types of $q$-series, 
namely those that Kac and Wakimoto referred to as ``multivariable Appell functions'', in a theory similar to the one he build around 
the level one Appell-Lerch sum in his doctoral thesis \cite{Zwegers2002}. We use his results to obtain the modularity of our characters.


In the second part of this article we investigate the connection of the product induced by the modular transformations of characters to the Grothendieck ring of characters. 
In rational conformal field theories it is straightforward to compute a product structure on the ring of characters using their modular properties. 
Namely, in the above example the character of the module $M_{\alpha,n}$ is 
\[
 \chi_{M_{\alpha,n}}(u;\tau)\ = \ z^{n/\alpha}q^{n^2/2\alpha^2} \frac{1}{\eta(\tau)} \theta_3(\alpha u+ n \tau;\alpha^2\tau).
\]
Then the {\emph{S-transformation}} is
\[
 \chi_{M_{\alpha,n}}\left(\frac{u}{\tau};-\frac{1}{\tau}\right)\ = \ \frac{1}{\sqrt{\alpha^2}} e^{\pi i u^2/\tau}\sum_{l=0}^{\alpha^2-1}  e^{-2\pi i nl/\alpha^2}\chi_{M_{\alpha,l}}(u;\tau).
\]
This transformation is used to define the numbers
\[
 S_{nl}\ = \ \frac{e^{-2\pi i nl/\alpha^2}}{\sqrt{\alpha^2}}\, .
\]
The matrix $S$ with entries $S_{nl}$ is called the ${\emph{S-matrix}}$. It is used to obtain the numbers
\begin{equation*}
 {N_{ab}}^c\ = \ \sum_{l=0}^{\alpha^2-1} \frac{S_{al}S_{bl}\bar{S_{lc}}}{S_{0l}}\ = \ \frac{1}{\sqrt{\alpha^2}}\sum_{l=0}^{\alpha^2-1} e^{-2\pi i l (a+b-c)}\ = \ \delta_{a+b,c}.
\end{equation*}
These numbers allow us to provide the space of characters with a product structure, which we denote by $\times_V$
\begin{equation*}
 \chi_{M_{\alpha,a}}(u;\tau)\times_{V}\chi_{M_{\alpha,b}}(u;\tau)\ :=\ \sum_{c=0}^{\alpha^2-1} {N_{ab}}^c\chi_{M_{\alpha,c}}(u;\tau)\ = \ \chi_{M_{\alpha,a+b}}(u;\tau).
\end{equation*}
This product coincides with the Grothendieck ring of characters, that is, the ring structure on the space of characters induced by the fusion product of modules
\begin{equation*}
\chi_{M_{\alpha,a}}(u;\tau)\times_{G}\chi_{M_{\alpha,b}}(u;\tau):=\chi_{M_{\alpha,a}\times_{f}M_{\alpha,b}}(u;\tau) \ = \ \chi_{M_{\alpha,a+b}}(u;\tau).
\end{equation*}

The Verlinde formula predicts that the ring of characters obtained from the modular transformations of characters is the Grothendieck ring of characters \cite{V}. In the case of rational conformal field theories this was proven by Huang \cite{H}. 

Here, we are interested in non-rational conformal field theories that contain a set of uncountable size of modules of the chiral algebra and that 
are quasi-rational in the sense that the fusion product of two indecomposable modules is a finite sum of indecomposable modules.
Then, one needs to establish the modularity of the characters and understand how these can be used to define a product structure on the space of characters before  postulating an analogue of the Verlinde formula.
There are examples of this type that are well-understood, namely the case of $\AKMSA{gl}{1}{1}$ \cite{CQS} and the case of an admissible but non-integer level $\AKMA{sl}{2}$ theory \cite{CR4}. 
In the first example characters are expressed in terms of Jacobi theta functions, while in the second
one characters are often divergent sums and need to be interpreted as distributions. 

We use the modular transformations of characters (Theorem \ref{thm:modprop}) to compute a product on the $\C$-span of characters. 
The fusion product of $\AKMSA{gl}{1}{1}$ is known and the fusion multiplicities are at most one \cite{CR1}.
Since the operator product algebra of $\AKMSA{gl}{1}{1}$-modules \cite{SS,CR3} is also known, it is clear whether there exist  intertwining operators 
between modules of the extended algebra. However, their multiplicities remain unknown.
In analogy to $\AKMSA{gl}{1}{1}$ it seems reasonable that these multiplicities are either zero or one. The comparison with the fusion product gives 

\begin{theorem}
The product on the space of characters obtained from the modular transformations coincides with the product of the Grothendieck ring of characters 
if and only if the fusion multiplicities of the extended algebra are at most one.
\end{theorem}

In contrast to the situation in rational conformal field theories we have to deal with two subtleties.
The modular $S$-transformation involves characters whose weight-labels are not necessarily real. We analytically continue them to real weight-labels. Furthermore, non-generic modules are quotients of reducible generic modules implying that there are linear dependencies
between characters of non-generic and generic modules. Using a regularization scheme we use these to express the characters of non-generic modules in terms of generic ones.

Constructions of this type have been used before. In the cases of $\AKMSA{gl}{1}{1}$ and admissible level $\AKMA{sl}{2}$, similar as in our example, the computation of the ring of characters from the modular transformations involves
the interpretation of characters of non-generic modules as infinite sums of characters of generic ones.
Further in the case of $\AKMSA{gl}{1}{1}$ a similar analytic continuation as in our example has been performed.

This work is organized as follows. In Section 2, we summarize the results needed on extended algebras of $\AKMSA{gl}{1}{1}$ of \cite{CR1}.
In Section 3, following Zwegers' work, we introduce higher-level Appell-Lerch sums and compute their modular and elliptic transformation properties. We also introduce a variant of the Mordell integral.
In Section 4, we compute the elliptic and modular transformations of the characters of the modules of the extended algebras. 
Especially, we give a representation theoretic meaning for the Mordell integral and we find that the $\C$-span of certain finite sets of characters of non-generic modules combined with sets of uncountable size of generic characters
are representations of $SL(2;\Z)$. 
In Section 5, we use the modular transformations to define a product structure on the space of characters and compare it
with the Grothendieck ring of characters as far as we know the latter.

\section*{Acknowledgments}
The authors thank Kathrin Bringmann, Stephan Ehlen, and David Ridout for helpful comments on an earlier draft of the paper.

\section{Extensions of the affine Lie superalgebra $\AKMSA{gl}{1}{1}$}\label{section:affinegl}

In this section, we summarize results of the affine Lie superalgebra $\AKMSA{gl}{1}{1}$ and the extended algebras of $\AKMSA{gl}{1}{1}$ discussed in \cite{CR1}.

The affine Lie superalgebra $\AKMSA{gl}{1}{1}$ is an infinite dimensional Lie superalgebra with generators $E_n,N_n,\psi^\pm_n$ ($n$ in $\Z$ ) and $K,d$, where $K$ is central and $d$ is a derivation, i.e.
\begin{equation}\label{eq:derivationgl11}
    [d,X_n]\ = \ nX_n\qquad\qquad\text{for} \ X \ \in \ \{E,N,\psi^\pm\}\, .
\end{equation}
The $\Psi^\pm_n$ are odd, while all others are even elements of the superalgebra.
The non-vanishing relations of the remaining generators are
\begin{equation*}
 \begin{split}
  [E_n,N_m]\ &= \ Kn\delta_{n+m,0},\qquad
 [N_n,\psi^\pm_m] \ = \ \psi^\pm_{n+m},\qquad
 \{\psi^-_n,\psi^+_m\}\ =\ E_{n+m}+Kn\delta_{n+m,0}\, .
 \end{split}
\end{equation*}
One usually considers representations with fixed $K$-eigenvalue $k\in\R$. 
As in the case of Lie algebras, one can then construct the generators of another infinite dimensional Lie algebra, the Virasoro algebra.
The generators are
\begin{eqnarray*}
 L_n & = &  \frac{1}{2k}\  (2 N_n E_0- E_n +
        \Psi^-_n \Psi^+_0 + \Psi^-_0 \Psi^+_n +\frac{1}{k} E_n E_0) \\
        & & +  \frac{1}{k} \sum_{m > 0} \, (E_{n-m} N_m +
    N_{n-m} E_m + \Psi^-_{n-m} \Psi^+_{m} - \Psi^+_{n-m} \Psi^-_m
    + \frac{1}{k} E_{n-m} E_m)
\end{eqnarray*} 
and they satisfy the relations of the Virasoro algebra
\begin{equation*}
 [L_n,L_m]\ = \ (n-m)L_{n+m}+\frac{c}{12}n(n^2-1)\delta_{n+m,0}
\end{equation*}
of central charge $c=0$. This construction is referred to as the Sugawara construction \cite{S}. 
We remark, that the level can be rescaled so as to normalize it to one as follows
\begin{equation*}
N_n \rightarrow N_n,\qquad E_n \rightarrow \frac{E_n}{k} ,\qquad \Psi^\pm_n \rightarrow \frac{\Psi^\pm_n}{\sqrt{k}}\, .
\end{equation*}
To simplify notation we set $k=1$ from now on.
The zero mode subalgebra is the finite dimensional Lie superalgebra $\SLSA{gl}{1}{1}$. It is usually called the horizontal subalgebra. Automorphisms that do not leave the horizontal subalgebra invariant are called spectral flow automorphisms in physics. For $\AKMSA{gl}{1}{1}$ they are $\sigma_m, m\in \Z$ and $\omega_\alpha, \alpha\in\C$, and they act as 
\begin{equation}\label{eq:auto}
\begin{split}
 \sigma_m(E_0)\ &= \ E_0+m, \qquad \sigma_m(\psi^\pm_n)\ = \ \psi^\pm_{n\pm m},\qquad\sigma_m(L_n)\ = \ L_n+mN_n,\\
 \omega_\alpha(N_0) \ &=\ N_0+\alpha ,\qquad \omega_\alpha(L_n) \ = \ L_n+\alpha E_n.
 \end{split} 
\end{equation}
and leave all other generators invariant. 
They become relevant in the discussion of elliptic transformations of characters.
Let $V$ be a super vector space and $\rho:\AKMSA{gl}{1}{1}\rightarrow \text{End}(V)$ a representation of $\AKMSA{gl}{1}{1}$.
We then define its character as
\begin{equation*}
\chi_\rho(u,v;\tau)\ = \ \text{str}_V(q^{\rho(L_0)}z^{\rho(N_0)}y^{\rho(E_0)})
\end{equation*}
with $q=e(\tau),z=e(u)$ and $y=e(v)$.
Representations were explicitly constructed in \cite{SS, L, CS, CR3} using free field realizations of the affine vertex superalgebra of $\SLSA{gl}{1}{1}$
and its modules. 
Irreducible representations fall into two families. 
Generically Verma modules are irreducible, in this case they are called typical. These are highest-weight representations characterized by the weight label $(e,n), e\notin\Z, n \in\C$.
We denote them by $\TypMod{n,e}$ and we require the highest-weight state to have parity $\lfloor e \rfloor$, where the number $\lfloor e \rfloor$ 
is the largest integer smaller or equal to the real part of $e$.
Then the character is
\begin{equation*}
\chi_{\TypMod{n,e}}(u,v;\tau) \ = \  i(-1)^{\lfloor e\rfloor}y^ez^nq^{\Delta_{n,e}}\frac{\theta_1(u;\tau)}{\eta^3(\tau)}
\end{equation*}
where the number
\begin{equation*}
\Delta_{n,e}\ = \ ne+\frac{e^2}{2}
\end{equation*}
is called the conformal dimension of the highest-weight vector of this module. It is required to be real. 
Representations of the non-generic family are irreducible quotients of Verma modules and are called atypical. 
They are characterized by the weight label $(\ell,n),\ell\in \Z, n\in\R$.
We denote them by $\AtypMod{n,\ell}$ with character
\begin{equation}\label{eq:atypgl11char}
\begin{split}
\chi_{\AtypMod{n,\ell}}(u,v;\tau) \ = \ -(-1)^{\ell}iy^\ell \frac{z^{n+1/2}q^{\Delta_{n,\ell}}}{1-zq^\ell}\frac{\theta_1(u;\tau)}{\eta^3(\tau)}\ \times\ 
\begin{cases} z^{-1/2}\, \ \,\quad\qquad &\text{if}\ \ell>0\,\\
\ \ 1 \, \ \,\quad\qquad &\text{if}\ \ell=0\,\\
z^{1/2}q^\ell\, \ \,\quad\qquad &\text{if}\ \ell<0\,.\\
 \end{cases}
\end{split}
\end{equation}
Here, we choose the parity of highest-weight states to agree with the parity given by the free field realizations of \cite{CR2,CR3}. 
Remark, that the vacuum module of the affine vertex algebra of $\SLSA{gl}{1}{1}$ is $\AtypMod{0,0}$.
This algebra can be extended by atypical modules. 
For this, note that the fusion product of atypical modules is
\begin{equation*}
\AtypMod{n,\ell} \times \AtypMod{n',\ell'}\ = \ \AtypMod{n+n'-\epsilon(\ell,\ell'),\ell+\ell'}
\end{equation*}
with $\epsilon(\ell,\ell')=\epsilon(\ell)+\epsilon(\ell')-\epsilon(\ell+\ell')$ and $\epsilon(\ell)$ being $1/2,0$ or $-1/2$ according as to whether $\ell\in\Z$ is positive, zero or negative, respectively. 
The fusion rules have been computed in \cite{CR1}, and the parity of the highest-weight state can be read off from the explicit free field realization of the algebra and its modules \cite{CR2,CR3}.
The extended algebras found in \cite{CR1} are of the following type. Let $\ell$ be a positive integer and $n$ such that
\begin{equation*}
\Delta_{n+1/2,\ell}\ = \ n\ell+\frac{\ell(\ell+1)}{2} \in \frac{1}{2}\Z_{>0}\, ,
\end{equation*}
then the fusion orbit of $\AtypMod{n+1/2,\ell}$ and its conjugate $\AtypMod{-n-1/2,-\ell}$
is
\begin{equation*}
\mathsf{A}_{0,0}^{n,\ell} \ = \ \bigoplus_{m\in\Z}\AtypMod{mn+\epsilon(m\ell),m\ell}\, . 
\end{equation*}
This is the vacuum module of a W-algebra, let us denote it by $\mathfrak W_{n,\ell}$, extending $\AKMSA{gl}{1}{1}$.
We require that 
\begin{equation*}
a\ = \ \frac{n}{\ell}\ \in\ \Z\, .
\end{equation*}
Let $0<\ell'<\ell$ and $m,m'\in\Z$, then
\begin{equation*}
\AtypMod{m\ell a+\epsilon(m\ell),m\ell} \times \AtypMod{(m'\ell+\ell') a+n'+\epsilon(m'\ell+\ell'),m'\ell+\ell'} \ = \ 
\AtypMod{((m+m')\ell+\ell')a+n'+\epsilon((m+m')\ell+\ell'),(m+m')\ell+\ell'}.
\end{equation*}
This means, that 
\begin{equation*}
\mathsf{A}_{n',\ell'}^{n,\ell} \ = \ \bigoplus_{m\in\Z}\AtypMod{(m\ell+\ell')a+n'+\epsilon(m\ell+\ell'),m\ell+\ell'}\,  
\end{equation*}
is an irreducible module for $\mathfrak W_{n,\ell}$. We call it atypical as it is composed of atypical irreducible $\AKMSA{gl}{1}{1}$ modules. 
Its character can be read off \eqref{eq:atypgl11char} and equals
\begin{equation*}
\chi_{\mathsf{A}_{n',\ell'}^{n,\ell}}(u,v;\tau) \ = \  -i\frac{\theta_1(u;\tau)}{\eta^3(\tau)}\sum_{m\in\Z}\frac{(-yz^a)^{m\ell+\ell'}z^{n'+\frac{1}{2}}q^{\frac{1}{2}(m\ell+\ell')((m\ell+\ell')(2a+1)+2n'+1)}}{1-zq^{m\ell+\ell'}}.
\end{equation*}
We are now going to combine generic $\AKMSA{gl}{1}{1}$ modules into irreducible modules of $\mathfrak W_{n,\ell}$.
The fusion rules for atypical-typical $\AKMSA{gl}{1}{1}$ modules are
\begin{equation*}
\AtypMod{n,\ell} \times \TypMod{n',e'}\ = \ \TypMod{n+n'-\epsilon(\ell),\ell+e'}\, .
\end{equation*}
This implies that the direct sum of typical modules 
\begin{equation*}
\mathsf{T}_{n',e'}^{n,\ell} \ = \ \bigoplus_{m\in\Z}\TypMod{mn+n',m\ell+e'}\,  
\end{equation*}
is an irreducible module of $\mathfrak W_{n,\ell}$ (for $e'\notin\Z$). As this is the generic case, we call it typical.
Its character is 
\begin{equation}\label{typchar}
\begin{split}
\chi_{\mathsf{T}^{n,\ell}_{n',e'}}(u,v;\tau)& \ = \ i(-1)^{\lfloor e'\rfloor}\frac{\theta_1(u;\tau)}{\eta^3(\tau)}y^{e'}z^{n'}q^{n'e'+e'^2/2}\sum_{m\in\mathbb Z}(-1)^{m\ell}y^{m\ell}z^{mn}q^{m^2/2(2n\ell+\ell^2)+m(ne'+n'\ell+\ell e')}\\
&= \  (-1)^{\lfloor e'\rfloor}\frac{\theta_1(u;\tau)}{\eta^3(\tau)}y^{e'-\ell/2}(-1)^{(\ell-1)/2}z^{n'-n/2}q^{\frac{1}{2}\bigl(2n'e'+e'^2-ne'-n'\ell-\ell e'+n\ell/2+\ell^2/4\bigr)}\times\\
&\qquad\times\theta_1\left(\frac{\ell-1}{2}+un+\ell v+\tau\left(ne'+n'\ell+e'\ell-n\ell-\frac{\ell^2}{2}\right);(2n\ell+\ell^2)\tau\right).
\end{split}
\end{equation}
Note, that modules are related as
\begin{equation}\label{periodicity}
\mathsf{A}^{n,\ell}_{n',\ell'}\ = \ \mathsf{A}^{n,\ell}_{n',\ell'+\ell}\qquad\text{and}\qquad\mathsf{T}^{n,\ell}_{n',e'}\ = \ \mathsf{T}^{n,\ell}_{n'+n,e'+\ell}\,.
\end{equation}
Typical and atypical characters are related as follows
\begin{equation}\label{atyptyp}
\chi_{\mathsf{A}^{n,\ell}_{n'+1,\ell'}}(u,v;\tau)-\chi_{\mathsf{A}^{n,\ell}_{n',\ell'}}(u,v;\tau)\ = \ \chi_{\mathsf{T}^{n,\ell}_{n'+\ell'n/\ell+1/2,\ell'}}(u,v;\tau)\, .
\end{equation}
This equation has to be taken with care as the modules $\mathsf{T}^{n,\ell}_{n',e'}$ are irreducible if and only if $e'\notin\Z$.  
The characters of $\mathsf{T}^{n,\ell}_{n',\ell'}$ with $\ell'\in\Z$ are characters of reducible but indecomposable modules.

The fusion rules are inherited from those of $\AKMSA{gl}{1}{1}$ up to multiplicities. This means, we know which fusion coefficients are non-zero, but we do not know
their precise multiplicities. We conjecture them to be as simple as possible.
\begin{conjecture}\label{conj}
The fusion rules of the extended algebra $\mathfrak W_{n,\ell}$ are
\begin{equation*}
\begin{split}
\mathsf{A}_{n',\ell'}^{n,\ell}\ \times_f \  \mathsf{A}_{n'',\ell''}^{n,\ell} \ 
&= \ \mathsf{A}_{n'+n'',\ell'+\ell'' }^{n,\ell} \\
\mathsf{A}_{n',\ell'}^{n,\ell}\ \times_f \  \mathsf{T}_{n'',e''}^{n,\ell} \ 
&= \ \mathsf{T}_{n'+n''+a\ell',\ell'+e'' }^{n,\ell}.
\end{split}
\end{equation*}
\end{conjecture}
We do not consider fusion products of typicals with themselves since they sometimes involve a third type of modules which does not provide additional information for the problem we consider.

Our aim is to obtain the modular and elliptic transformation properties of these two families of characters and to
use them to compute a Verlinde formula that is in agreement with Conjecture \ref{conj}.

\section{Appell-Lerch sums and the Mordell integral}

We study higher-level Appell-Lerch sums and the Mordell integral. First we recall important work of Zwegers \cite{Zwegers2002} on (level one) Appell-Lerch sums.

 For $u\in\C$ and $\tau\in\h$ we let as before
\begin{equation}\label{eq:theta1}
 \theta_1(u;\tau)\ :=\ -i \sum_{n\in\Z} (-1)^n e^{\pi i (n+\frac12)^2 \tau+2\pi i u (n+\frac12)}.
\end{equation}
For $u,v\in \C\setminus(\Z\tau+\Z)$ and $\tau\in\h$ Zwegers defines the normalized Appell-Lerch sum
\begin{equation*}
\mu(u,v;\tau)\ :=\ \frac{e^{2\pi i u}}{\vartheta(v;\tau)} \sum_{n\in\Z}\frac{(-1)^ne^{\pi i\tau n(n+1)}e^{2\pi i nv}}{1-e^{2\pi i n \tau+2\pi i u}},
\end{equation*}
where $\vartheta(v;\tau)\ = \-\theta_1(v;\tau)$.
Here we consider $\vartheta(v;\tau)\mu(u,v;\tau)$, namely we let
\begin{equation*}
 A_1(u,v;\tau) \ = \ e^{2\pi i u}\sum_{n\in\Z}\frac{(-1)^ne^{\pi i n(n+1)\tau+2\pi i n v}}{1-e^{2\pi i n\tau+2\pi i u}}\ = \ z\sum_{n\in\Z} \frac{(-1)^n q^{\frac12n(n+1)}y^n}{1-zq^n},
\end{equation*}
where as before $z:=e^{2\pi i u}, y:=e^{2\pi i v}, q:=e^{2\pi i\tau}$.
In order to consider its modular transformation properties we introduce the Mordell integral
\begin{equation*}
 h(u;\tau)\ :=\ \int_\R\frac{e^{\pi i \tau x^2-2\pi zx}}{\cosh \pi x}dx.
\end{equation*}
Then we have (see Proposition 1.4 in \cite{Zwegers2002})
\begin{proposition}\label{trafoA1}
Assume notation as above. Then
\begin{align*}
\theta_1(u+1;\tau)\ &= \ -\theta_1(u;\tau)\\
 \theta_1(u+\tau;\tau)\ &= \ -e^{-\pi i \tau -2\pi i u}\theta_1(u;\tau)\\
\theta_1(u;\tau+1)\ &= \ e^{\frac{\pi i }{4}} \theta_1(u;\tau)\\
\theta_1\left(\frac{u}{\tau};-\frac{1}{\tau}\right)\ &= \ -i\sqrt{-i\tau}e^{\pi i u^2/\tau}\theta_1(u;\tau).
\end{align*}
For  $A_1(u,v;\tau)$ we obtain the following elliptic transformation properties
\begin{equation*}
\begin{split}
 A_1(u+1,v;\tau)\ &=\ -A_1(u,v;\tau)\\
 A_1(u,v+1;\tau)\ &=\ A_1(u,v;\tau)\\ 
 A_1(u+\tau,v;\tau)\ &=\ -e^{2\pi i(u-v+\tau/2)}A_1(u,v;\tau)-ie^{\pi i(u-v+3\tau/4)}\theta_1(v;\tau)\\
 A_1(u,v+\tau;\tau)\ &= \ e^{-2\pi i u}A_1(u,v;\tau)+ie^{-\pi i(u+v+\tau/4)}\theta_1(v;\tau)\, .
\end{split}
\end{equation*}
and the modular transformation properties
\begin{equation}\label{leveloneS}
\begin{split}
 A_1(u,v;\tau+1)\ &=\ A_1(u,v;\tau)\\
A_1\left(\frac{u}{\tau},\frac{v}{\tau};-\frac{1}{\tau}\right)\ &= \ \tau e^{-\pi i(u^2-2uv)/\tau}\left(A_1(u,v;\tau)+\frac{i}{2}\theta_1(v;\tau)h(u-v;\tau)\right)\, .
\end{split}
\end{equation}
\end{proposition}

For $K\in\mathbb{N}$ the level $K$ Appell-Lerch sum is defined as
\begin{equation*}
A_K(u,v;\tau)\ = \ z^{K/2}\sum_{n\in\mathbb Z} \frac{(-1)^{Kn}q^{Kn(n+1)/2}y^n}{1-zq^n}.
\end{equation*}
This function was first introduced and studied in \cite{STT}, see also \cite{Zwegers1}.

\begin{remark}
Note that
\begin{equation}\label{chiasAppell}
\chi_{\mathsf{A}_{n',0}^{a,1}}(u,v;\tau)\ = \ -i\frac{\theta_1(u;\tau)}{\eta^3(\tau)} z^{n'-a} A_{1+2a}\left(u,v+au+\tau(n'-a);\tau\right).
\end{equation}
\end{remark}

The higher level Appell-Lerch sum can be expressed in terms of the level one sum as follows (also see \cite{Zwegers1}).
\begin{proposition}\label{AKinA1}
The level $K$ Appell-Lerch sum satisfies
\begin{equation}\label{eq:rel1}
A_K(u,v;\tau)\ = \ \sum_{m=0}^{K-1}z^m A_1\left(Ku,v+m\tau +\frac{K-1}{2};K\tau\right)
\end{equation}
and
\begin{equation}\label{eq:rel2}
A_K(u,v;\tau)\ = \ K^{-1}z^{(K-1)/2}\sum_{m=0}^{K-1} A_1\left(u,\frac{v}{K}+\frac{m}{K} +\tau\frac{K-1}{2K};\frac{\tau}{K}\right)\, .
\end{equation}
\end{proposition}
\begin{proof}
Rewriting the right-hand side of \eqref{eq:rel1} yields
\begin{equation*}
\begin{split}
\sum_{m=0}^{K-1}z^m A_1\left(Ku,v+m\tau +\frac{K-1}{2};K\tau\right) \ &= \\
= \ z^{K/2}\sum_{n\in\mathbb Z}(-1)^{Kn}&y^nq^{Kn(n+1)/2}\sum_{m=0}^{K-1}\frac{(zq^n)^m}{1-(zq^n)^K}  = \  A_K(u,v;\tau)\, ,
\end{split}
\end{equation*}
so we obtain the first statement. Computing
\begin{equation*}
\begin{split}
K^{-1}z^{(K-1)/2}\sum_{m=0}^{K-1} A_1\left(u,\frac{v}{K}-\frac{m}{K} +\tau\frac{K-1}{2K};\frac{\tau}{K}\right)\ &= \\
  = \ K^{-1}z^{K/2}\sum_{n\in\mathbb Z} \frac{(-1)^{n}q^{n^2/(2K)+n/2}y^{n/K}}{1-zq^{n/K}}\sum_{m=0}^{K-1}&e^{2\pi i mn/K}   \ = \ A_K(u,v;\tau)\,  
\end{split}
\end{equation*}
gives the second statement. \eop
\end{proof}

This allows us to compute the elliptic and modular transformation properties (also see Theorem 2.2 of \cite{Zwegers1}).
\begin{corollary}\label{trafosAK}
The level $K$ Appell-Lerch sum satisfies the following elliptic transformation properties
\begin{equation}\label{AKtrafos}
\begin{split}
 A_K(u+1,v;\tau)\ &=\ (-1)^K A_K(u,v;\tau)\\
 A_K(u+\tau,v;\tau)\ &=\ (-1)^K z^Ky^{-1} q^{K/2}A_K(u,v;\tau)+ \\
 &\qquad - i^K z^{K/2}y^{-1/2}q^{3K/8} \sum_{m=0}^{K-1} z^mq^{\frac{m}{2}} \theta_1\left(v+m\tau+\frac{K+1}{2};K\tau\right)\\
 A_K(u,v+1;\tau)\ &=\ A_K(u,v;\tau)\\
 A_K(u,v+\tau;\tau)\ &=\ z^{-1}A_K(u,v;\tau)+i^K z^{\frac{K}{2}-1}y^{-\frac{1}{2}}q^{-\frac{K}{8}}\theta_1\left(v+\frac{K+1}{2};K\tau\right)\,.
\end{split}
\end{equation}
Moreover, the level $K$ Appell-Lerch sum satisfies the following modular transformation properties
\begin{equation}\label{AKtau1}
   A_K(u,v;\tau+1)\ =\ A_K(u,v;\tau)
\end{equation}
and
\begin{equation}\label{AKS}
\begin{split}
A_K\Bigl(\frac{u}{\tau}&,\frac{v}{\tau};-\frac{1}{\tau}\Bigr)  \ = \  \tau e^{-\pi i(Ku^2-2vu)/\tau}\Bigl(A_K(u,v;\tau)+\\ 
&\ \ +\frac{i}{2K}z^{(K-1)/2}\sum_{m=0}^{K-1}
\theta_1\Bigl(\frac{K-1}{2K}\tau+\frac{v}{K}-\frac{m}{K};\frac{\tau}{K}\Bigr)h\Bigl(u-\frac{K-1}{2K}\tau-\frac{v}{K}+\frac{m}{K};\frac{\tau}{K}\Bigr)\Bigr)\, .
\end{split}
\end{equation}
or rewritten
\begin{equation}\label{AKS2}
\begin{split}
A_K\Bigl(\frac{u}{\tau},\frac{v}{\tau}&;-\frac{1}{\tau}\Bigr)\ = \ \tau e^{-\pi i(Ku^2-2vu)/\tau}\Bigl(A_K(u,v;\tau)+\\ 
&\quad\ \ \ +\frac{i}{2}\sum_{m=0}^{K-1}z^m
\theta_1\left(v+m\tau-\frac{K-1}{2};K\tau\right)h\left(Ku-v-m\tau+\frac{K-1}{2};K\tau\right)\Bigr)\, .
\end{split}
\end{equation}
\end{corollary}

\begin{proof}
 Combining Proposition \ref{AKinA1} and the transformation properties of the level one Appell function (Proposition \ref{trafoA1}) directly yields \eqref{AKS} and \eqref{AKS2}. The transformation properties \eqref{AKtau1} and in the first and third equation of \eqref{AKtrafos} are obvious.
For the second equation of \eqref{AKtrafos} observe that
\begin{align*}
 A_K(u+\tau,v;\tau)&\ = \ z^{K/2}q^{K/2}\sum_{n\in\Z}\frac{(-1)^{Kn}q^{\frac{K}{2}n(n+1)}y^n}{1-zq^{n+1}}\ = \
\\
&\ = \ (-1)^K z^{K/2}q^{K/2}y^{-1}\sum_{n\in\Z}\frac{(-1)^{Kn}q^{\frac{K}{2}n(n-1)}y^n}{1-zq^{n}}(1-(zq^n)^K+(zq^n)^K).
\end{align*}
For the last equation of \eqref{AKtrafos} observe that
\[
 A_K(u,v+\tau;\tau)\ = \ z^{K/2}\sum_{n\in\Z}\frac{(-1)^{Kn}q^{\frac{K}{2}n(n+1)}y^nq^n}{1-zq^{n}}(1-(zq^n)^{-1}+(zq^n)^{-1}).
\]
\eop\end{proof}

\begin{remark}
 In \cite{Zwegers1} Zwegers shows that the higher level Appell-Lerch sums transform as Jacobi forms if completed with certain nonholomorphic correction terms. Proposition \ref{AKinA1} and Corollary \ref{trafosAK} are also (indirectly) proven in that paper.
\end{remark}


We define the variant of the Mordell integral
\begin{equation*}
h_s(u;\tau)\ =\begin{cases}
 \ \int_{\mathbb R} \frac{q^{x^2/2}z^{ix}}{\text{cosh}{\pi (x-is)}}dx\, & \text{for } 0\leq |s|\leq 1, s\neq\pm 1/2
 \\
 \ \int_{\mathbb R-i\epsilon} \frac{q^{x^2/2}z^{ix}}{\text{cosh}{\pi (x\mp i/2)}}dx\, & \text{for } s =\pm 1/2,
 \end{cases}
\end{equation*}
where $\epsilon>0$.
We compute 
\begin{equation*}
\int_{\mathbb R+is} \frac{q^{x^2/2}z^{ix}}{\text{cosh}{\pi (x-is)}}dx\ = \ q^{-s^2/2}z^{-s}\int_{\mathbb R} \frac{q^{x^2/2}z^{ix}e^{-2\pi\tau s x}}{\text{cosh}{\pi x}}dx
\ = \ q^{-s^2/2}z^{-s}h(u+s\tau;\tau)\, .
\end{equation*}
Hence, we have for $-\frac12\leq s<\frac12$ that
\begin{equation*}
\begin{split}
h_s(u;\tau)- q^{-s^2/2}z^{-s}h(u+s\tau;\tau) \ &= \ \Bigl(\int_{\mathbb R}-\int_{\mathbb R+is}\Bigr) \frac{q^{x^2/2}z^{ix}}{\text{cosh}{\pi (x-is)}}dx\\
&= 0.
\end{split}
\end{equation*}
We summarize
\begin{proposition}\label{trafoMordell}
For $-1/2 \leq s<1/2$ the Mordell integral has the following transformation property 
\begin{equation*}
h(u+s\tau) \ = \ q^{s^2/2}z^{s}h_s(u;\tau).
\end{equation*}
\end{proposition}
We remark that $h_1(u;\tau)=-h(u;\tau)$. The transformation properties of the Mordell integral for integral $s$ can be found in \cite{Zwegers1}.

\section{Transformation Properties of Characters}

In this section, we compute the modular and elliptic transformation properties of the characters.
We start with atypical characters. Recall, that we required $a=n/\ell\in\Z$, $n\in\Z$.
\begin{proposition}
The elliptic transformation properties of the character $\chi_{\mathsf{A}^{n,\ell}_{n',\ell'}}(u,v;\tau)$ are
\begin{equation*}
\begin{split}
\chi_{\mathsf{A}^{n,\ell}_{n',\ell'}}(u+1,v;\tau)\ &= \ e^{2\pi i(a\ell'+n')}
\chi_{\mathsf{A}^{n,\ell}_{n',\ell'}}(u,v;\tau)\\
\chi_{\mathsf{A}^{n,\ell}_{n',\ell'}}(u+\tau,v;\tau)\ &= \ y^{-1}\chi_{\mathsf{A}^{n,\ell}_{n'-a-1,\ell'+1}}(u,v;\tau)\\
\chi_{\mathsf{A}^{n,\ell}_{n',\ell'}}(u,v+1;\tau)\ &= \ \chi_{\mathsf{A}^{n,\ell}_{n',\ell'}}(u,v;\tau)\\
\chi_{\mathsf{A}^{n,\ell}_{n',\ell'}}(u,v+\tau;\tau)\ &= \ z^{-1}\chi_{\mathsf{A}^{n,\ell}_{n'+1,\ell'}}(u,v;\tau)\, .
\end{split}
\end{equation*}
More generally, we have for $\alpha\in\R$
\begin{equation*}
\chi_{\mathsf{A}^{n,\ell}_{n',\ell'}}(u,v+\alpha\tau;\tau)\ = \ z^{-\alpha}\chi_{\mathsf{A}^{n,\ell}_{n'+\alpha,\ell'}}(u,v;\tau)\, .
\end{equation*}
\end{proposition}
\begin{proof}
Using the definition of $\chi_{\mathsf{A}^{n,\ell}_{n',\ell'}}$ the proof of these equalities is a straightforward computation.
\end{proof}
\begin{remark}
The elliptic transformations are related to the automorphisms of $\AKMSA{gl}{1}{1}$ in \eqref{eq:auto}.
Denote the induced representation of a representation $\rho$ of $\AKMSA{gl}{1}{1}$ by $\sigma^*(\rho)$ for an automorphism $\sigma$.
The automorphism $\omega_\alpha$ \eqref{eq:auto} is rather trivial in the sense, that it only changes the weight-label $n$
\begin{equation*}
(\omega_\alpha)^*(\AtypMod{n,\ell}) \ = \ \AtypMod{n-\alpha,\ell}\, .
\end{equation*}
It thus lifts to the following map of atypical $\mathfrak W_{n,\ell}$ modules
\begin{equation*}
(\omega_\alpha)^*(\mathsf{A}^{n,\ell}_{n',\ell'}) \ = \ \mathsf{A}^{n,\ell}_{n'-\alpha,\ell'}\, .
\end{equation*}
This is reflected in the elliptic transformation
\begin{equation*}
z^{-\alpha}\chi_{\mathsf{A}^{n,\ell}_{n',\ell'}}(u,v-\alpha\tau;\tau)\ = \ \chi_{\mathsf{A}^{n,\ell}_{n'-\alpha,\ell'}}(u,v;\tau) \ = \ \chi_{(\omega_\alpha)^*(\mathsf{A}^{n,\ell}_{n',\ell'})}(u,v;\tau)\, .
\end{equation*}
The automorphism $\sigma_m$ is less trivial. It relates the $\AKMSA{gl}{1}{1}$ modules as follows \cite{CR1}
\begin{equation*}
(\sigma_m)^*(\AtypMod{n,\ell}) \ = \ \AtypMod{n+m+\epsilon(\ell-m)-\epsilon(\ell),\ell-m}\, .
\end{equation*}
It thus lifts to the following map of atypical $\mathfrak W_{n,\ell}$ modules
\begin{equation*}
(\sigma_m)^*(\mathsf{A}^{n,\ell}_{n',\ell'}) \ = \ \mathsf{A}^{n,\ell}_{n'+m+ma,\ell'-m}\, .
\end{equation*}
This is reflected in the elliptic transformation for $\ell'-m>0$
\begin{equation*}
y^{-m}\chi_{\mathsf{A}^{n,\ell}_{n',\ell'}}(u-m\tau,v;\tau)\ = \ \chi_{\mathsf{A}^{n,\ell}_{n'+m+ma,\ell'-m}}(u,v;\tau) \ = \ \chi_{(\sigma_m)^*(\mathsf{A}^{n,\ell}_{n',\ell'})}(u,v;\tau)\, .
\end{equation*}
Also note that the automorphisms that lift to automorphisms of the W-algebra are $\sigma_\ell\circ \omega_{n+\ell}$.
\end{remark}

Elliptic transformation properties of typical characters are similar and can be read off \eqref{typchar} directly
\begin{proposition}
The elliptic transformation properties of the character $\chi_{\mathsf{T}^{n,\ell}_{n',e'}}(u,v;\tau)$ are
\begin{equation*}
\begin{split}
\chi_{\mathsf{T}^{n,\ell}_{n',e'}}(u+1,v;\tau)\ &= \ e^{2\pi i(n'-1/2)}
\chi_{\mathsf{T}^{n,\ell}_{n',e'}}(u,v;\tau)\\
\chi_{\mathsf{T}^{n,\ell}_{n',e'}}(u+\tau,v;\tau)\ &= \ y^{-1}\chi_{\mathsf{T}^{n,\ell}_{n'-1,e'+1}}(u,v;\tau)\\
\chi_{\mathsf{T}^{n,\ell}_{n',e'}}(u,v+1;\tau)\ &= \ e^{2\pi i e'} \chi_{\mathsf{T}^{n,\ell}_{n',e'}}(u,v;\tau)\\
\chi_{\mathsf{T}^{n,\ell}_{n',e'}}(u,v+\tau;\tau)\ &= \ z^{-1}\chi_{\mathsf{T}^{n,\ell}_{n'+1,e'}}(u,v;\tau)\, .
\end{split}
\end{equation*}
More generally, we have for $\alpha\in\R$,
\begin{equation*}
\chi_{\mathsf{T}^{n,\ell}_{n',e'}}(u,v+\alpha\tau;\tau)\ = \ z^{-\alpha}\chi_{\mathsf{T}^{n,\ell}_{n'+\alpha, e'}}(u,v;\tau)\, .
\end{equation*}
\end{proposition}

We turn to the modular properties of characters.
Here the idea is to reduce the computation to the family $\ell=1$. 
We need the following observation.
\begin{lemma}\label{chiunity}
Let $\xi_\ell=e^{2\pi i /\ell}$ and $I:=\Z \cap \left[n,n+\ell-1\right]$, then the atypical characters satisfy 
\begin{equation*}
\begin{split}
 \chi_{\mathsf{A}^{n,\ell}_{n',t}}(u,v;\tau)\ &= \ \frac{1}{\ell}\sum_{s\in I}\xi_\ell^{-ts} \chi_{\mathsf{A}^{a,1}_{n',0}}\left(u,v+\frac{s}{\ell};\tau\right)\quad\text{and}\\
  \chi_{\mathsf{A}^{a,1}_{n',0}}\left(u,v+\frac{t}{\ell};\tau\right)\ &= \ \sum_{s\in I} \xi_\ell^{ts}\chi_{\mathsf{A}^{n,\ell}_{n',s}}(u,v;\tau)\, .
\end{split}
 \end{equation*}
The typical ones satisfy
 \begin{equation*}
\begin{split}
\chi_{\mathsf{T}_{n'+tn/\ell,e'+t}^{n,\ell}}(u,v;\tau)\ &= \ \frac{1}{\ell}\sum_{s\in I}\xi_\ell^{-st} e^{-2\pi i e's/\ell}\chi_{\mathsf{T}_{n',e'}^{a,1}}\left(u,v+\frac{s}{\ell};\tau\right) \quad\text{and} \\
\chi_{\mathsf{T}_{n',e'}^{a,1}}\left(u,v+\frac{s}{\ell};\tau\right)\ &= \ e^{2\pi i e's/\ell} \sum_{t\in I}\xi_\ell^{st} \chi_{\mathsf{T}_{n'+tn/\ell,e'+t}^{n,\ell}}(u,v;\tau)\,. \\
\end{split}
\end{equation*}
 \end{lemma}
 \begin{proof} The statement follows using results on sums over $\ell-$th roots of unity.
 \eop \end{proof}
So it suffices to deduce the modular transformation properties of $\chi_{\mathsf{A}^{a,1}_{n',0}}(u,v+s/\ell;\tau)$.

\begin{lemma}\label{lm:trafoatyp}
Let $-\ell/2\leq s,t <\ell/2$ and define
\begin{equation}\label{weightlabels}
a_m(x)\ = \ \frac{a(a-m)}{2a+1}+\frac{a}{2}+ix(a+1) \qquad,\qquad e_m(x) \ = \ \frac{(a-m)}{2a+1}+\frac{1}{2}-ix
\end{equation}
for $x\in\R$. Then
\begin{equation*}
\begin{split}
\chi_{\mathsf{A}^{a,1}_{t/\ell,0}}\left(\frac{u}{\tau},\frac{v}{\tau}+\frac{s}{\ell};-\frac{1}{\tau}\right) \ &= \
 e^{2\pi i uv/\tau}\Bigl[\chi_{\mathsf{A}^{a,1}_{s/\ell,0}}\Bigl(u,v-\frac{t}{\ell};\tau\Bigr)+\\
 &\qquad +\frac{1}{2}\sum_{m=0}^{2a}\int_{\mathbb R}dx\frac{\chi_{\mathsf{T}^{a,1}_{a_{m-s/\ell}(x),e_{m-s/\ell}(x)}}\Bigl(u,v-\frac{t}{\ell};\tau\Bigr)}{\text{cosh}\Bigl(\pi\Bigl(x+i\frac{a-m+\frac{s}{\ell}}{2a+1}\Bigr)\Bigr)}\ \ \Bigr]
\end{split}
\end{equation*} 
\end{lemma} 
\begin{proof}
We let $K:=2a+1$. By \eqref{chiasAppell} we see that
\begin{equation}\label{trafochi1}
\begin{split}
& \chi_{\mathsf{A}^{a,1}_{t/\ell,0}}\left(\frac{u}{\tau},\frac{v}{\tau}+\frac{s}{\ell};-\frac{1}{\tau}\right) \ = \ i \frac{1}{\tau} e^{\pi i u^2/\tau}e^{2\pi i u (\frac{t}{\ell}-a)/\tau}\frac{\theta_1(u;\tau)}{\eta^3(\tau)}\times
\\
&\qquad\times  \tau e^{-\pi i (Ku^2 -2u(v+au+\tau(\frac{s}{\ell}-a)+a-\frac{t}{\ell}) )/\tau} \left[A_{K}\left(u,v+au+\tau(\frac{s}{\ell}-a)+a-\frac{t}{\ell};\tau\right)\right.+
\\
&\left. +\ \frac{i}{2} \sum_{m=0}^{2a} z^m \theta_1\left(v+au+\tau(\frac{s}{\ell}-a+m)-\frac{t}{\ell};K\tau\right)h\left(u(1+a)-v-\tau(\frac{s}{\ell}-a+m)+\frac{t}{\ell};K\tau\right)\right]
.
\end{split}
\end{equation}
We have
\begin{equation*}
 e^{-\pi i u^2/\tau}e^{2\pi i u (\frac{t}{\ell}-a)/\tau}e^{-\pi i (Ku^2 -2u(v+au+\tau(\frac{s}{\ell}-a)+a-\frac{t}{\ell}) )/\tau}\ = \ e^{2\pi i uv/\tau}z^{\frac{s}{\ell}-a}.
\end{equation*}
Thus, the expression in \eqref{trafochi1} simplifies to
\begin{equation*}
\begin{split}
&\chi_{\mathsf{A}^{a,1}_{t/\ell,0}}\left(\frac{u}{\tau},\frac{v}{\tau}+\frac{s}{\ell};-\frac{1}{\tau}\right) \ = \
e^{2\pi i uv/\tau} \chi_{\mathsf{A}^{a,1}_{s/\ell,0}}\left(u,v-\frac{t}{\ell};\tau\right)
 - \frac{1}{2} \frac{\theta_1(u;\tau)}{\eta^3(\tau)}e^{2\pi iuv/\tau}z^{\frac{s}{\ell}-a} \times
\\
&\times\sum_{m=0}^{2a} z^m \theta_1\left(v+au+\tau\left(\frac{s}{\ell}-a+m\right)-\frac{ts}{\ell};K\tau\right)h\left(u(1+a)-v-\tau\left(\frac{s}{\ell}-a+m\right)+\frac{t}{\ell};K\tau\right).
\end{split}
\end{equation*}
Using Proposition \ref{trafoMordell} we can rewrite the latter expression as the sum of integrals over the typical characters for the given parameters.\eop
\end{proof} 
We compute the modular properties of the typical characters. More specifically, we are interested in
\begin{equation*}
\chi_{\mathsf{T}^{a,1}_{a_{m-s/\ell}(x),e_{m-s/\ell}(x)}}(u,v;\tau)\ = \ x^k\frac{\theta_1(u;\tau)}{\eta^3(\tau)}y^{b-ix}
z^{ab+(a+1)ix}q^{\frac{K}{2}\bigl(x^2+b^2\bigr)}\theta_1(v+ua+Kb\tau;K\tau)
\end{equation*}
with parameters
\begin{equation*}
b\ = \ \frac{a-m+s/\ell}{K} \qquad,\qquad K\ =\ 2a+1\, .
\end{equation*}
We have 
\begin{lemma}\label{lm:trafotypchar}
Let $0\leq m\leq 2a$ and $-\frac{\ell}{2}\leq s\leq \frac{\ell}{2}$, then the typical character satisfies
\begin{align*}
&\chi_{\mathsf{T}^{a,1}_{a_{m-\frac{s}{\ell}}(x),e_{m-\frac{s}{\ell}}(x)}}\left(\frac{u}{\tau},\frac{v}{\tau}+\frac{t}{\ell};-\frac{1}{\tau}\right) \ = \
\\
&\quad\quad\quad \ = \
e^{2\pi i uv/\tau} \sum_{m'=0}^{K-1} \int_{\mathbb R}dw\,   e^{2\pi i A_{m'}(w)}\,
\chi_{\mathsf{T}^{a,1}_{a_{m'-\frac{t}{\ell}}(w),e_{m'-\frac{t}{\ell}}(w)}}\left(u,v-\frac{s}{\ell};\tau\right)
\end{align*}
where
\begin{equation}\label{def:Am}
\begin{split}
A_{m'}(w) \ &= \ -\frac{ixt}{\ell}-\frac{iws}{\ell}-wKx +
\frac{st}{\ell^2K}-\frac{(m'+1/2)(m1/2)}{K}+\frac{1}{4}\, .
\end{split}
\end{equation}
\end{lemma}
\begin{proof}
Let $\tilde{y}:=e^{2\pi i v/\tau}, \tilde{z}:=e^{2\pi i u/\tau}$, and $ \tilde{q}:=e^{-2\pi i /\tau}$.
We compute 
\begin{equation}\label{eq:Stypical}
\begin{split}
\chi_{\mathsf{T}^{a,1}_{a_{m-s/\ell}(x),e_{m-s/\ell}(x)}}\left(\frac{u}{\tau},\frac{v}{\tau}+\frac{t}{\ell};-\frac{1}{\tau}\right)\ &= \
\frac{1}{-i\tau}\frac{\theta_1(u;\tau)}{\eta^3(\tau)}\tilde{y}^{-ix}\tilde{z}^{ix(a+1)}\tilde{q}^{\frac{K}{2}x^2}e^{\frac{\pi i}{K\tau}(v+ua)^2}e^{\frac{\pi i}{\tau}u^2}\times\\
&\quad\times e^{2\pi x t/\ell}z^{\frac{at}{\ell K}} y^{\frac{t}{\ell K}}
 q^{\frac{t^2}{2K \ell^2}}e^{-\pi i Ku'^2/\tau}\theta_1\Bigl(\frac{K}{\tau}u';-\frac{K}{\tau}\Bigr)\, 
\end{split}
\end{equation}
with $$u'\ = \ \frac{v}{K}+\frac{ua}{K}-b+\frac{t\tau}{K\ell}\, .$$ 
Recall the standard Gauss integral 
\begin{equation}\label{eq:gauss}
\int_{\mathbb R} dx e^{-\alpha x^2+\beta x} \ = \ \sqrt{\frac{\pi}{\alpha}}e^{\beta^2/(4\alpha)}\, \quad \Re(\alpha)>0 .
\end{equation}
Thus, we have that
\begin{equation*}
\frac{1}{\sqrt{-i\tau K}}\tilde{y}^{-ix}\tilde{z}^{ix(a+1)}\tilde{q}^{\frac{K}{2}x^2} \ = \ e^{-\frac{\pi i}{\tau K}(v-(a+1)u)^2}\int_{\mathbb R} dw\,q^{\frac{K}{2}w^2}z^{iw(a+1)}y^{-iw}e^{-2\pi i wKx} \,  .
\end{equation*}
Together with 
\begin{equation}\label{eq:thetascale}
\begin{split}
 \theta_1\left(u;\frac{\tau}{K}\right)\ = \ \sum_{n=0}^{K-1} & q^{(n-(K-1)/2)^2/2K}e^{2\pi i (n-(K-1)/2)(u+1/2)}\times 
\\
& \qquad\qquad\times\theta_1\left(Ku+\tau\left(n-\frac{K-1}{2}\right)+\frac{K-1}{2};K\tau\right)
\end{split}
\end{equation}
this implies
\begin{equation*}
\begin{split}
\chi_{\mathsf{T}^{a,1}_{a_{m-\frac{s}{\ell}}(x),e_{m-\frac{s}{\ell}}(x)}}&\left(\frac{u}{\tau},\frac{v}{\tau}+\frac{t}{\ell};-\frac{1}{\tau}\right)\ = \\
&=\ e^{2\pi i \frac{uv}{\tau}} \sum_{m'=0}^{K-1} \int_{\mathbb R}dw\,   e^{2\pi i A_{m'}(w)}\,
\chi_{\mathsf{T}^{a,1}_{a_{m'-\frac{t}{\ell}}(w),e_{m'-\frac{t}{\ell}}(w)}}\left(u,v-\frac{s}{\ell};\tau\right)
\end{split}
\end{equation*}
where we changed summation to $m'=2a-n$ in \eqref{eq:thetascale}. Here, $A_{m'}(w)$ is given as in \eqref{def:Am}.
\eop
\end{proof}


Define the sets 
\begin{equation*}
\begin{split}
S\ &=\ \{s\in\Z|-\ell/2\leq s<\ell/2\}\quad\text{and}\quad
 M\ =\ \left\{m\in\frac{1}{\ell}\Z|-\frac{\ell}{2}(2a+1)<m\leq \frac{\ell}{2}(2a+1)\right\}\, .
 \end{split}
 \end{equation*}

\begin{theorem}\label{thm:modprop}
Let $s,s',t,t'\in S$, $r,c\in M$, $x\in\R$ and define the {\emph{S-matrices}}
\begin{equation}\label{Scompl}
\begin{split}
S_{(t/\ell,t')}^{(s/\ell,s')}\ &= \ \frac{1}{\ell}\xi_\ell^{-t's-s't}\, , \\
S_{(t/\ell,t')}^{(a_r(x),e_r(x))}\ &= \ \frac{1}{2\ell}\frac{e^{2\pi i (t'r-e_r(x)t/\ell)}}{\text{sin}(\pi e_r(x))}\, , \\
S_{(a_r(x),e_r(x))}^{(a_c(w),e_c(w))}\ &= \ \frac{1}{\ell}(-1)^{\lfloor e_r(x)\rfloor+\lfloor e_c(w)\rfloor}e^{2\pi i (e_r(x)Ke_c(w)+e_r(x)(c-2a-1/2)+e_c(w)(r-2a-1/2)+1/4)}\, .
\end{split}
\end{equation}
Then we have the following modular transformation properties
\begin{equation*}
\begin{split}
\chi_{\mathsf{A}^{n,\ell}_{t/\ell,t'}}\left(\frac{u}{\tau},\frac{v}{\tau};-\frac{1}{\tau}\right)\ &= \ e^{2\pi i uv/\tau} 
\Bigl[ \sum_{s,s'\,\in\, S} S_{(t/\ell,t')}^{(s/\ell,s')}\chi_{\mathsf{A}^{n,\ell}_{s/\ell,s'}}(u,v;\tau) +\\
&\qquad +\sum_{r\in M}\int_{\mathbb R}dx\, S_{(t/\ell,t')}^{(a_r(x),e_r(x))} \chi_{\mathsf{T}^{n,\ell}_{a_{r}(x),e_{r}(x)}}\Bigl(u,v;\tau\Bigr)\ \ \Bigr]\, , \\
\chi_{\mathsf{A}^{n,\ell}_{t/\ell,t'}}(u,v;\tau+1)\ &= \ \xi_{\ell}^{tt'}\chi_{\mathsf{A}^{n,\ell}_{t/\ell,t'}}(u,v;\tau)\, , \\
\chi_{\mathsf{T}^{n,\ell}_{a_r(x),e_r(x)}}\left(\frac{u}{\tau},\frac{v}{\tau};-\frac{1}{\tau}\right)\ &= \ e^{2\pi i uv/\tau} \sum_{c\in M}\int_\R dw\,  S_{(a_r(x),e_r(x))}^{(a_c(w),e_c(w))}
\chi_{\mathsf{T}^{n,\ell}_{a_c(w),e_c(w)}}(u,v;\tau)\, , \\
\chi_{\mathsf{T}^{n,\ell}_{a_r(x),e_r(x)}}(u,v;\tau+1)\ &= \ e^{\pi i\left(\frac{(a-r)^2}{K^2}+(a-r)+Kx^2+\frac{a}{2}+\frac14\right)}\chi_{\mathsf{T}^{n,\ell}_{a_r(x),e_r(x)}}(u,v;\tau)\, . \\
\end{split}
\end{equation*}
\end{theorem}
\begin{proof}
The formulas for the $T$-transformations follow directly. For the $S$-transformations we first employ Lemma \ref{chiunity} and then use Lemma \ref{lm:trafoatyp} respectively Lemma \ref{lm:trafotypchar}. The ``reverse'' application of Lemma \ref{chiunity} yields the results where we use the periodicity of $\cosh$ and change the summation index to $r:=m-s/\ell-s'K$ for the formula for $\chi_{\mathsf{A}^{n,\ell}_{t/\ell,t'}}\left(\frac{u}{\tau},\frac{v}{\tau};-\frac{1}{\tau}\right)$. For $\chi_{\mathsf{T}^{n,\ell}_{a_r(x),e_r(x)}}\left(\frac{u}{\tau},\frac{v}{\tau};-\frac{1}{\tau}\right)$ we set
\[
r \ = \ m-s/\ell-s'K\qquad\text{and}\qquad   c\ = \ m'-t/\ell-t'K.
\]
Again, the repeated use of Lemma \ref{chiunity} together with a change of the summation index yields the result.
\eop
\end{proof}
\begin{remark}
The S-matrices satisfy the following periodicity conditions
\begin{equation*}
\begin{split}
S_{(t/\ell,t')}^{(s/\ell,s')}\ &= \ S_{(t/\ell,t'+m\ell)}^{(s/\ell,s'+m'\ell)}\, , \\
S_{(t/\ell,t')}^{(a_r(x),e_r(x))}\ &= \ S_{(t/\ell,t'+m\ell)}^{(a_{r+m'\ell(2a+1)}(x),e_{r+m'\ell(2a+1)}(x))}\ = \ S_{(t/\ell,t'+m\ell)}^{(a_r(x)+m'n,e_r(x)+m'\ell)}\, , \\
S_{(a_r(x),e_r(x))}^{(a_c(w),e_c(w))}\ &= \ S_{(a_{r+m'\ell(2a+1)}(x),e_{r+m'\ell(2a+1)}(x))}^{(a_{c+m\ell(2a+1)}(w),e_{c+m\ell(2a+1)}(w))}\ = \ 
S_{(a_r(x)+m'n,e_r(x)+m'\ell)}^{(a_c(x)+mn,e_c(x)+m\ell)}\, ,
\end{split}
\end{equation*}
for $m,m'\in\Z$. These periodicities nicely correspond to the character identities \eqref{periodicity}.
\end{remark}

\section{The Ring of Characters}

In rational conformal field theories the modular transformation properties induce a ring structure on the vector space of characters via the Verlinde formula.
This ring structure gives the Grothendieck ring of characters \cite{H}.

In this section, we explain how to use the mock modular transformation properties in Theorem \ref{thm:modprop} to provide the vector space of characters with a product structure.
We show that this product coincides with the corresponding product in the Grothendieck ring of characters if and only if Conjecture \ref{conj} holds. 
While the use of the Verlinde formula is straight-forward in rational conformal field theories, we have to deal with two subtleties here. Namely, we have to analytically continue the
S-matrix and we have to view characters of atypical modules as (regularized) sums of characters of typical ones. Similar techniques were applied in the computation of the Verlinde formula of $\AKMSA{gl}{1}{1}$ \cite{CR1} and a commutant subalgebra of one of our theories \cite{CR4}.

We are interested in characters with real weight labels, while the modular transformations of last section involved not necessarily real weights.
Hence, we define the {\emph{S-matrix}} as an analytic continuation to real weight labels of \eqref{Scompl}.
Recall that $a_r(x)=2a-r-e_r(x)(a+1)+1/2$ \eqref{weightlabels}.
\begin{definition}
Let $s,s',t,t'\in S$, $r,r'\in M$, $e,e'\in\R$ and let $m=2a-r,m'=2a-r'$, then we define
\begin{equation*}
\begin{split}
{S_{aa}}_{(t/\ell,t')}^{(s/\ell,s')}\ &= \ \frac{1}{\ell}\xi_\ell^{-t's-s't}\, , \\
{S_{at}}_{(t/\ell,t')}^{(m-(a+1)e+1/2,e)}\ &= \ \frac{1}{2\ell}\frac{e^{2\pi i (t'm-et/\ell)}}{\text{sin}(\pi e)}\, , \\
{S_{tt}}_{(m-(a+1)e+1/2,e)}^{(m'-(a+1)e'+1/2,e')}\ &= \ \frac{1}{\ell}(-1)^{\lfloor e\rfloor+\lfloor e'\rfloor}e^{2\pi i (eKe'-e(m'-1/2)-e'(m-1/2)+1/4)}\, .
\end{split}
\end{equation*}
We call $S_{{tt}}$ the \emph{typical} (and $S_{{aa}}$ the \emph{atypical}) {\emph{S-matrix}}.
\end{definition}
\begin{proposition}
The typical and atypical S-matrices are symmetric 
\begin{equation*}
{S_{tt}}_{(m-(a+1)e+1/2,e)}^{(m'-(a+1)e'+1/2,e')}\ = \ {S_{tt}}_{(m'-(a+1)e'+1/2,e')}^{(m-(a+1)e+1/2,e)}\qquad,\qquad
{S_{aa}}_{(t/\ell,t')}^{(s/\ell,s')} \ = \ {S_{aa}}_{(s/\ell,s')}^{(t/\ell,t')}
\end{equation*}
and unitary
\begin{equation*}
\begin{split}
\sum_{m'\in M}\int de'\,
{S_{tt}}_{(m-(a+1)e+1/2,e)}^{(m'-(a+1)e'+1/2,e')}{\overline{S}_{tt}}_{(m'-(a+1)e'+1/2,e')}^{(m''-(a+1)e''+1/2,e'')}\ &= \ \delta_{m,m''}\delta(e-e'')\, ,\\
\sum_{s',s\in S} {S_{aa}}_{(t/\ell,t')}^{(s/\ell,s')}{\overline{S}_{aa}}_{(s/\ell,s')}^{(r/\ell,r')} \ &= \ \delta_{t,r}\delta_{t',r'}\, .
\end{split}
\end{equation*}
The bar denotes complex conjugation.
\end{proposition}
\begin{proof}
The first statement on the symmetry is clear. For the second statement observe that
\begin{align}\label{comp:unitary}
\sum_{m'\in M}\int de'\,
\notag&{S_{tt}}_{(m-(a+1)e+1/2,e)}^{(m'-(a+1)e'+1/2,e')}{\overline{S}_{tt}}_{(m'-(a+1)e'+1/2,e')}^{(m''-(a+1)e''+1/2,e'')} \, =
\\
&\quad =\frac{1}{\ell^2} \sum_{m'\in M}\int de'\, (-1)^{\lfloor e\rfloor+\lfloor e''\rfloor} e^{2\pi i e'(eK-e''K-m+m'')}e^{2\pi i m'(e''-e)}e^{\pi i (e-e'')}.
\end{align}
Now 
\[
\int de'\,e^{2\pi i e'(eK-e''K-m+m'')}=0
\]
unless $eK-e''K=m''-m$. So we can consider 
\[
\sum_{m'\in M} e^{2\pi i m'(m-m'')/K},
\]
which equals $\delta_{m,m''}$ by standard results on roots of unity.
Then the delta distribution can only be non-zero if $e-e'\in\Z$  in which case the phase $e^{\pi i(e-e')}(-1)^{\lfloor e\rfloor+\lfloor e'\rfloor }$
becomes one. 
 So \eqref{comp:unitary} equals
\[
\delta_{m,m''}\delta(e-e'').
\]

Moreover, we have
\begin{align*}
\sum_{s',s\in S} {S_{aa}}_{(t/\ell,t')}^{(s/\ell,s')}{\overline{S}_{aa}}_{(s/\ell,s')}^{(r/\ell,r')} &=\frac{1}{\ell^2} \sum_{s',s\in S}  \xi_{\ell}^{-t's-s't}\bar{\xi_{\ell}^{-s'r-r's}}
\\
&= \frac{1}{\ell^2} \sum_{s',s\in S} e^{\frac{2\pi i }{\ell} (s(r'-t')+s'(r-t))}.
\end{align*}
Using standard results on roots of unity we see that this equals $\delta_{t,r}\delta_{t',r'}$.
\eop
\end{proof}
\begin{definition}
Define the structure constants for $t,t'\in S$,$e,e'\in\R$ and $m,m'\in M$ as
\begin{equation*}
\begin{split}
&{N_{(t/\ell,t'),(m-(a+1)e+1/2,e)}}^{(m'-(a+1)e'+1/2,e')}\ = \\
 &\qquad\qquad= \ \sum_{k\in M}\int_\R d x\, 
\frac{{S_{at}}_{(t/\ell,t')}^{(k-(a+1)x+1/2,x)}{S_{tt}}_{(m-(a+1)e+1/2,e)}^{(k-(a+1)x+1/2,x)}{\overline{S}_{tt}}_{(k-(a+1)x+1/2,x)}^{(m'-(a+1)e'+1/2,e')}}{{S_{at}}_{(0,0)}^{(k-(a+1)x+1/2,x)}}\,.
\end{split}
\end{equation*}
\end{definition}
\begin{proposition}
The structure constants satisfy
\begin{equation*}
{N_{(t/\ell,t'),(m-(a+1)e+1/2,e)}}^{(m'-(a+1)e'+1/2,e')}\ = \ \delta(e-e'+(m'-m-t/\ell)/K)\delta_{m'=m+t/\ell+t'K\,\text{mod}\,\ell K}\, .
\end{equation*}
\end{proposition}
\begin{proof}
The statement follows from the following computation
\begin{equation*}
\begin{split}
&{N_{(t/\ell,t'),(m-(a+1)e+1/2,e)}}^{(m'-(a+1)e'+1/2,e')}\ = \\
 &\qquad\qquad= \ \frac{1}{\ell^2}\sum_{k\in M}\int_\R d x\, 
e^{2\pi i xK(e-e'+(m'-m-t/\ell)/K)}e^{2\pi ik(e'-e-t')}e^{\pi i(e-e')}(-1)^{\lfloor e\rfloor+\lfloor e'\rfloor }\\
&\qquad\qquad= \ \frac{1}{K\ell^2}\sum_{k\in M}\delta(e-e'+(m'-m-t/\ell)/K)e^{2\pi ik(e'-e-t')}e^{\pi i(e-e')}(-1)^{\lfloor e\rfloor+\lfloor e'\rfloor }\\
&\qquad\qquad= \ \frac{1}{K\ell^2}\sum_{k\in M}\delta(e-e'+(m'-m-t/\ell)/K)e^{2\pi i\frac{k}{K}(m'-m-t'K)}e^{\pi i(e-e')}(-1)^{\lfloor e\rfloor+\lfloor e'\rfloor }\\
&\qquad\qquad=\ \delta(e-e'+(m'-m)/K)\delta_{m'=m+t/\ell+t'K\,\text{mod}\,\ell K}\, .
\end{split}
\end{equation*}
The delta distribution can only be non-zero if $e-e'\in\Z$,  in which case the phase $e^{\pi i(e-e')}(-1)^{\lfloor e\rfloor+\lfloor e'\rfloor }$
becomes one. \eop
\end{proof}
Since $m,m'\in M$ and $t,t'\in S$, we may choose an $\epsilon\in\{-1,0,1\}$ such that $m+t/\ell+t'K+\epsilon\ell K\in M$.
We now define and compute an action of atypical characters on typical ones by
\begin{align}
\nonumber
\chi_{\mathsf{A}^{n,\ell}_{t/\ell,t'}}(u,v;\tau)\times \chi_{\mathsf{T}^{n,\ell}_{m-(a+1)e,e}}(u,v;\tau)\ &= \ \sum_{m'\in M}\int_\R de' {N_{(t/\ell,t'),(m-(a+1)e+1/2,e)}}^{(m'-(a+1)e'+1/2,e')}\\ \label{prodat}
&= \  \chi_{\mathsf{T}^{n,\ell}_{m+t/\ell-(a+1)e+t'a+\epsilon\ell  a,e+t'+\epsilon\ell}}(u,v;\tau)\\
&= \  \chi_{\mathsf{T}^{n,\ell}_{m+t/\ell-(a+1)e+t'a,e+t'}}(u,v;\tau)\, .\nonumber
\end{align}
For the last equality we used $a\ell=n$ and \eqref{periodicity}. Remark that this action extends to an action of the $\C$-span of atypical characters on the $\C$-span of typical characters by linearity.
In order to obtain a product structure on atypical characters we use the relation to typical ones \eqref{atyptyp}, namely
\begin{equation*}
\chi_{\mathsf{A}^{n,\ell}_{s/\ell,s'}}(u,v;\tau)-\chi_{\mathsf{A}^{n,\ell}_{s/\ell+m+1,s'}}(u,v;\tau) \ =\ \sum_{i=0}^m \chi_{\mathsf{T}^{n,\ell}_{s/\ell+i+s'a+1/2,s'}}(u,v;\tau)
\end{equation*}
for all $m>0,m\in\Z$
and combine this with \eqref{prodat} to get 
\begin{equation}\label{aam}
\begin{split}
\chi_{\mathsf{A}^{n,\ell}_{t/\ell,t'}}(u,v;\tau)\times (\chi_{\mathsf{A}^{n,\ell}_{s/\ell,s'}}(u,v;\tau)-&\chi_{\mathsf{A}^{n,\ell}_{s/\ell+m+1,s'}}(u,v;\tau)) \ =\\
 &= \ \sum_{i=0}^m \chi_{\mathsf{T}^{n,\ell}_{(s+t)/\ell+i+(s'+t')a+1/2,s'+t'}}(u,v;\tau)\\
 &= \ \chi_{\mathsf{A}^{n,\ell}_{(s+t)/\ell,s'+t'}}(u,v;\tau)-\chi_{\mathsf{A}^{n,\ell}_{(s+t)/\ell+m+1,s'+t'}}(u,v;\tau)\\
\end{split}
\end{equation}
We would like to take the limit $m\rightarrow\infty$. Here a regularization is needed. For $\epsilon>0$ we let
\begin{equation*}
\chi^\epsilon_{\mathsf{A}^{n,\ell}_{n',\ell'}}(u,v;\tau)\ = \ q^{\epsilon n'^2}\chi_{\mathsf{A}^{n,\ell}_{n',\ell'}}(u,v;\tau)\, .
\end{equation*}
The product of regularized characters
\begin{equation}\label{regprod}
\begin{split}
\chi^{\epsilon}_{\mathsf{A}^{n,\ell}_{t/\ell,t'}}(u,v;\tau)\times (\chi^{\epsilon}_{\mathsf{A}^{n,\ell}_{s/\ell,s'}}(u,v;\tau)-&\chi^{\epsilon}_{\mathsf{A}^{n,\ell}_{s/\ell+m+1,s'}}(u,v;\tau)) \ =\\
 &= \ \chi^{\epsilon}_{\mathsf{A}^{n,\ell}_{(s+t)/\ell,s'+t'}}(u,v;\tau)-\chi^{\epsilon}_{\mathsf{A}^{n,\ell}_{(s+t)/\ell+m+1,s'+t'}}(u,v;\tau)\\
\end{split}
\end{equation}
agrees with \eqref{aam} in the limit $\epsilon\rightarrow 0$.
The regularized character vanishes at infinity
\begin{equation*}
 \lim_{n'\rightarrow\infty}\chi^{\epsilon}_{\mathsf{A}^{n,\ell}_{n',\ell'}}(u,v;\tau)\ = \ 0\,\qquad \text{for\ all}\ \epsilon >0\, .
\end{equation*}
Hence, in the limit $m\rightarrow\infty$, the regularized product \eqref{regprod} of \eqref{aam} becomes 
\begin{equation*}
\begin{split}
\chi^{\epsilon}_{\mathsf{A}^{n,\ell}_{t/\ell,t'}}(u,v;\tau)\times \chi^{\epsilon}_{\mathsf{A}^{n,\ell}_{s/\ell,s'}}(u,v;\tau)\ = \ 
\chi^{\epsilon}_{\mathsf{A}^{n,\ell}_{(s+t)/\ell,s'+t'}}(u,v;\tau)\, .
\end{split}
\end{equation*}
We summarize
\begin{proposition}
The modular S-transformation gives the following products for $t,t',s,s'\in S,e\in\R,m\in M$
\begin{equation*}
\begin{split}
\chi_{\mathsf{A}^{n,\ell}_{t/\ell,t'}}(u,v;\tau)\times \chi_{\mathsf{T}^{n,\ell}_{m-(a+1)e,e}}(u,v;\tau)\ &= \
\chi_{\mathsf{T}^{n,\ell}_{m+t/\ell-(a+1)e+t'a,e+t'}}(u,v;\tau)\\
\chi^{\epsilon}_{\mathsf{A}^{n,\ell}_{t/\ell,t'}}(u,v;\tau)\times \chi^{\epsilon}_{\mathsf{A}^{n,\ell}_{s/\ell,s'}}(u,v;\tau)\ &= \ 
\chi^{\epsilon}_{\mathsf{A}^{n,\ell}_{(s+t)/\ell,s'+t'}}(u,v;\tau)\, .
\end{split}
\end{equation*}
\end{proposition}
Comparing with Conjecture \ref{conj}, we observe
\begin{corollary}
Assume that Conjecture \ref{conj} is correct, i.e. that fusion multiplicities are at most one. 
Then the product structure from the modular transformations coincides with the corresponding product in the Grothendieck ring of characters
\begin{equation*}
\begin{split}
\chi_{\mathsf{A}^{n,\ell}_{t/\ell,t'}}(u,v;\tau)\times \chi_{\mathsf{T}^{n,\ell}_{m-(a+1)e,e}}(u,v;\tau)\ &= \
\chi_{\bigl(\mathsf{A}^{n,\ell}_{t/\ell,t'}\times_f \mathsf{T}^{n,\ell}_{m-(a+1)e,e}\bigr)}(u,v;\tau)\\
\chi^{\epsilon}_{\mathsf{A}^{n,\ell}_{t/\ell,t'}}(u,v;\tau)\times \chi^{\epsilon}_{\mathsf{A}^{n,\ell}_{s/\ell,s'}}(u,v;\tau)\ &= \ 
\chi^{\epsilon}_{\bigl(\mathsf{A}^{n,\ell}_{t/\ell,t'}\times_f\mathsf{A}^{n,\ell}_{s/\ell,s'}\bigr)}(u,v;\tau)\, .
\end{split}
\end{equation*}
\end{corollary}

\end{document}